\documentclass[12pt,reqno]{amsart}
\usepackage{graphicx}
\usepackage{amsmath}
\usepackage{amssymb}
\usepackage{amscd}
\usepackage{hhline}
\usepackage{mathrsfs}

\newtheorem{theorem}{Theorem}

\newtheorem{theoremb}{Theorem}
\newtheorem{theoremc}{Theorem}
\newtheorem{theoremd}{Theorem}
\newtheorem{theoreme}{Theorem}
\newtheorem{dfn}[theoremb]{Definition}
\newtheorem{rk}[theoremc]{Remark}

\newtheorem{lem}[theoremd]{Lemma}
\newtheorem{examp}[theoreme]{Example}
\newtheorem{prop}[theorem]{Proposition}
\newcommand\bib[1]{\bibitem[#1]{#1}}

\newcommand\1{{\bf 1}}
\renewcommand\a{\alpha}
\renewcommand\b{\beta}

\newcommand\C{{\mathbb C}}
\newcommand\Cc{{\let\mathcal\mathscr\mathcal C}}

\renewcommand\d{\delta}
\newcommand\D{{\mathcal D}}

\newcommand\E{\mathcal{E}}

\newcommand\g{\gamma}

\renewcommand\l{\lambda}

\newcommand\La{\Lambda}

\newcommand\oo{\omega}
\newcommand\op[1]{\mathop{\rm #1}\nolimits}
\newcommand\ot{\otimes}
\newcommand\p{\partial}

\newcommand\R{{\mathbb R}}
\renewcommand\t{\tau}

\newcommand\vp{\varphi}

\newcommand\z{\sigma}

\begin{document}

 \title[Symmetry, compatibility and exact solutions]{Symmetry, compatibility\\ and exact solutions of PDEs}
 \author{Boris Kruglikov}
 \date{}
 \address{Institute of Mathematics and Statistics, University of Troms\o, Troms\o\ 90-37, Norway. \quad
E-mail: boris.kruglikov@uit.no.}

 \vspace{-14.5pt}
 \begin{abstract}
We discuss various compatibility criteria for overdetermined systems of PDEs generalizing the approach
to formal integrability via brackets of differential operators. Then we give sufficient conditions
that guarantee that a PDE possessing a Lie algebra of symmetries has invariant solutions.
 % These conditions hold for generic=nondegenerate pairs (equation,symmetry)
Finally we discuss models of equations with large symmetry algebras, which eventually
lead to integration in closed form.
 \end{abstract}

 \maketitle

%%%%%%%%%%%%%%%%%%%%%%%%%%%%%%%%%%%%%%%%%%%%%%%%%%%%%%%%%%%%%%%%%%%%%%%%%%%%
%0%
\section*{Introduction}

Overdetermined systems of PDEs always have compatibility conditions. If these are
satisfied, the system is called formally integrable (we assume regularity throughout the paper)
and we can formally parametrize the space of solutions \cite{C$_3$,BCG$^3$,KL$_2$}.

For Frobenius type systems the compatibility conditions are just equalities of mixed derivatives.
For instance (we denote the partial derivatives as usual by indices) the system $\E$ on $\R^2(x^1,x^2)$
 $$
\{F_{ij}:u_{ij}=f_{ij}(x,u,\p u)\,|\,i,j=1,2\}
 $$
has compatibility conditions
 $$
\left\{\begin{array}{l} \D_1(F_{12})=\D_2(F_{11})\\
\D_1(F_{22})=\D_2(F_{12})\end{array}\right. \Leftrightarrow
\left\{\begin{array}{l} \D_1(f_{12})=\D_2(f_{11})\\
\D_1(f_{22})=\D_2(f_{12})\end{array}\right.\!\!\!\mod\E.
 $$
In general this is wrong.
For example, consider the following system $\E$ on $\R^4(x^0,x^1,x^2,x^3)$
 \begin{equation}\label{NoC}
 \left\{\begin{array}{l}
F_1:\ u_{13} = u_{22}+f_1(x,u,\p u),\\
F_2:\ u_{12} = u_{03}+f_2(x,u,\p u),\\
F_3:\ u_{02} = u_{11}+f_3(x,u,\p u).
 \end{array}\right.
 \end{equation}
Equalities of the mixed derivatives do not yield compatibility here,
because the first compatibility condition involving only two equations is of order 3
(the corresponding syzygy operator, which is the Mayer bracket described in Theorem \ref{MB},
is of the second order), and it depends on derivatives of $f$ up to the 2nd order
(for more details see Section \ref{S1}).

There exists however a compatibility condition of order 2 (the syzygy operator involves
only first derivatives of $f$; $F_i$ below denotes the difference of the left and
right hand sides in the respective equation):
 \begin{equation}\label{Dsy}
\D_1(F_1)+\D_2(F_2)+\D_3(F_3)\equiv0\mod\E.
 \end{equation}

In Section \ref{S1} we discuss how to calculate the compatibility conditions.
This is related in the next Section to classical and higher symmetries, and we discuss
the question of existence of invariant solutions.

Consider a PDE $\E$ and a subalgebra $\mathcal{G}\subset\op{sym}(\E)$; at this point we restrict
to the local extrinsic (point, contact or higher) symmetries.
In general it is not true that $\E$ has $\mathcal{G}$-invariant solutions, as was noticed
in \cite{K$_1$}. The simplest counter-example constitute linear systems with the symmetry
being shift by a nonzero solution.

Another example is given by a pair of constant coefficients non-homogeneous linear scalar PDEs.
They commute provided there are no zero order terms, and so the equations are symmetries of each other.
Generically the PDEs are compatible, but this does not happen always
(e.g. the system $\{u_{xx}=1,u_x=0\}$ has no solutions at all).
A similar story happens for matrix differential operators and symmetries.

In Section \ref{S2} we demonstrate that this is a consequence of either degeneracy or
higher dimensionality of the symmetry algebra compared to the amount of independent variables.
We will prove that under certain genericity conditions the symmetry is compatible with
the equation. This yields (a similar result is proven by another approach in \cite{IV}):
 \begin{quote}
{\em Provided the symbols of $\E$ and $\mathcal{G}$ are generic,
the system $\E$ has $\mathcal{G}$-invariant solutions.}
 \end{quote}
 %  $$
 % \text{\em If the symbols of }\E\text{\em and }G\text{ are generic, }\E\text{ has }G\text{-invariant solutions.}
 %  $$

In Section \ref{S3} we discuss implications that existence of a large symmetry group has on
the solution space of a PDE system. At this point we need to consider intrinsic symmetries.
Relations between extrinsic and intrinsic symmetries are given by Lie-B\"acklund type theorems,
see \cite{AI,KLV,AKO,AK}, and this relates this integrability problem with what is discussed above
(see also \cite{L$_1$,K$_2$,Ga,BCA} for applications of symmetries and generalizations).

We will concentrate on exact solvability of ODEs and PDEs and discuss relations
with Darboux integrability via examples. Informally a lot of symmetry implies exact integrability,
more precisely this holds true for maximal symmetric models (can fail for sub-maximal cases).

We will briefly discuss some symmetric models. Relations of integrability to transformations
and differential substitutions in PDEs is central in \cite{L$_2$,K$_3$,K$_4$} for the case when the system $\E$
depends on 1 function of 1 argument (so called Lie class $\oo=1$). This applies in the other
cases too, but will be considered elsewhere.

The paper is organized so that all sections can be read independently.

\medskip

{\textsc{Acknowledgement.}} I thank V.\, Lychagin and A.\, Prasolov for helpful discussions.
Some calculations in Section \ref{S3} were performed with \textsc{Maple} package \textsc{DifferentialGeometry} by
I.\, Anderson. I am grateful for organizers of the conference SPT-2011 for hospitality.

%%%%%%%%%%%%%%%%%%%%%%%%%%%%%%%%%%%%%%%%%%%%%%%%%%%%%%%%%%%%%%%%%%%%%%%%%%%%
%1%
\section{Compatibility, differential syzygies and brackets}\label{S1}

An overdetermined system of differential equations can be viewed
geometrically as a finite sequence of submanifolds $\E_k\subset
J^k(\pi)$ in jets with $\E_{k-1}^{(1)}\supset\E_k$, where $\pi:E\to
M$ is a vector bundle and $\E_{k-1}^{(1)}$ denotes the prolongation
(locus of the derivatives of functions specifying $\E_{k-1}$).

Formal integrability is equivalent to the claim that all projections
$\pi_{k+1,k}:\E_{k+1}\to\E_k$ are submersions. Then we can define
$\E_\infty=\lim\limits_{\leftarrow}\E_k$.

Let $T$ be the model tangent space for independent variables and $N$
the model tangent space for dependent variables. Define the symbols of order
$k$ as $g_k=\op{Ker}(d\pi_{k,k-1}:T\E_k\to T\E_{k-1})\subset S^kT^*\otimes N$ and
let $g=\oplus g_k$ be the symbol bundle over $\E$.
For $k$ greater than the (maximal) order $l$ of $\E$ we can define
$g_k=S^{k-l}T^*\ot g_l\cap S^kT^*\ot N$.

The Spencer $\delta$-cohomology group $H^{*,*}(\E)$ is the cohomology of the complex
$g_*\ot\La^*(T^*)$ with symbolic de Rham differential over $\E$ \cite{S}.

The formal theory of differential equations identifies
obstructions to formal integrability (compatibility conditions) as certain elements
$W_k\in H^{k-1,2}(\E)$ (structural functions or Weyl tensors).
For geometric problems $H^{*,2}(\E)$ is the space of curvature/torsion tensors \cite{KL$_3$}.

Another way to look at this is to regard compatibility
conditions as differential syzygies, i.e. relations between generators of $\E$.

Denote generators of this system by ${\bf F}=\{F_i\}_{i=1}^r$, so that $\E=\{F_1[u]=0,\dots,F_r[u]=0\}$.
Then the symbol spaces can be expressed through the symbols of linearizations of these
(nonlinear) differential operators
(we will use the notation $\z$ for the symbol)
 $$
g=\op{Ker}\{\sigma(\ell_{F_1}),\dots,\sigma(\ell_{F_r})\}\subset
ST^*\otimes N.
 $$

It turns out that $g^*=\oplus g_k^*$ is an
$ST$-module, where the latter is viewed as the algebra of
polynomials on $T^*$. This is called the symbolic module
of the system $g^*=\mathcal{M}_\E$ (also known as the symbol module \cite{S}).

 \begin{dfn}
A differential operator $G$ from the left differential ideal $\langle{\bf F}\rangle$ is called
a differential syzygy for the system $\E$ if its symbol $\sigma(\ell_G)$ is an ordinary
(algebraic) syzygy for the symbolic module $\mathcal{M}_\E$.
 \end{dfn}

If $\E$ is linear we can consider $\langle{\bf F}\rangle$ as the
left module over the algebra $\op{Diff}(\1,\1)$ of scalar linear
differential operators (on the trivial rank 1 bundle $\1$ over $M$).
In the nonlinear case, this should be changed to the algebra
$\Cc\op{Diff}(\1,\1)=C^\infty(J^\infty\pi)\otimes_{C^\infty(M)}\op{Diff}(\1,\1)$
of $\Cc$-differential operators,
see \cite{KLV}: $\langle{\bf F}\rangle=\langle\Delta\circ F_i\,|\,\Delta\in\Cc\op{Diff}(\1,\1)\rangle$.

Thus to every differential syzygy there corresponds an algebraic syzygy, but
because the operator of taking symbol is not injective, the reverse problem
 $$
\text{syzygy}\stackrel{\mathfrak{q}}\rightsquigarrow\text{differential
syzygy},
 $$
is not uniquely solvable. Moreover it is not possible to construct $\mathfrak{q}$ as a homomorphism,
but we would like to have a map $\mathfrak{q}$, defined on a finite generating set of
syzygies, with nice algebraic properties (extension to the whole syzygy module is similar to quantization
theories). This will be done below for the class of (generalized) complete intersections.

Let $f_i=\sigma(\ell_{F_i})\in ST\otimes N^*$ be symbols of the
differential operators defining $\E$ and let $\sum g_if_i=0$ be a
syzygy, with $g_i\in ST$ being some polynomials on $T^*$ with
$k_i=\op{deg}(f_i)=\op{ord}(F_i)$.

Choose any $\Cc$-differential operators $G_i$ with $\z(G_i)=g_i$.
Then the corresponding differential syzygy has the form
 $$
\nabla=\sum G_i\circ F_i.
 $$
The order of this operator is $k-1$, where $k=\op{max}_i\{\op{ord}(G_i)+k_i\}$ is called the order
of the syzygy.

Non-uniqueness comes through the lower order terms in $G_i$. So the
class $\nabla\,\op{mod}\mathcal{J}_{k-1}({\bf F})$ is well defined,
where
 $$
\mathcal{J}_t({\bf F})=\{\sum Q_i\circ F_i:\op{ord}(Q_i)\le t-k_i\}.
 $$

Denote by $[S]$ the equivalence class $S\,\op{mod}\mathcal{J}_{k-1}({\bf
F})$ of the differential syzygy $S=\mathfrak{q}(s)$, where $k$ is
order of the algebraic syzygy $s$ (notice that this equivalence class is independent
of the choice of $\mathfrak{q}$ so far it is right-inverse to the symbol map $\z$).

 \begin{theorem}[\cite{KL$_4$}]\label{DS}
An overdetermined PDE system $\E$ is formally integrable iff
for a basis $\{s_j\}$ of algebraic syzygy the corresponding classes
of differential syzygies vanish $[S_j]=0$.
 \end{theorem}

This follows from the fact that the Spencer
$\delta$-cohomology group $H^{k-1,2}(\E)$ equals the corresponding
graded second cohomology group for projective resolution of the
symbolic module $\mathcal{M}_\E$, i.e. it can be enumerated via a
basis of algebraic syzygies.

There are two basic approaches to construct the arrow $\mathfrak{q}$:

1. Construct differential syzygies successively in order $k$, i.e.
according to passage $\E_k\dashrightarrow\E_{k+1}$. This corresponds to
prolongation-projection approach having origin in \'E.\,Cartan's equivalence
method.

2. Successive identification of differential syzygies involving two operators
$F_i,F_j$, then involving three etc. This represents compatibility as Massey products
\cite{KL$_4$} and is related to deformation of the symbolic module $\mathcal{M}_\E$
(or to the corresponding noncommutative D-module $\E^*$).

We will elaborate the first idea for PDE systems with
nice characteristic variety $\op{Char}^\C(\E)=\{\xi\in \mathbb{P}^\C
T^*:\op{rank}[\z(\ell_F)(\xi)]<m=\dim N\}$.

\smallskip

Now I want to present the explicit form of compatibility conditions for
certain generic overdetermined systems of PDEs. We start with the scalar case $m=1$
(single $u$); $n=\dim M$ is arbitrary.

The Jacobi bracket $\{F,G\}$ of scalar (non-linear) differential
operators $F,G\in\op{diff}(\1,\1)$ is defined via the linearization
operator as follows:
 $$
\{F,G\}=\ell_F\,G-\ell_G\,F.
 $$
If $\op{ord}F=k$, $\op{ord}G=l$, then $\op{ord}\{F,G\}=k+l-1$.

Let $(x^i,u_\z)$, $1\le i\le n$, $\z=(i_1,\dots,i_n)$,
be the "canonical" coordinates on the jet-space,
$\D_i=\p_{x^i}+\sum u_{\z+1_i}\p_{u_\z}$ the total derivative
w.r.t. $x^i$ and  let $\D_\z=\D_1^{i_1}\dots\D_n^{i_n}$ be the
operator of higher total derivative. We can express the Jacobi bracket as
 $$
\{F,G\}= \sum \D_\z(F)\p_{u_\z}(G)-\D_\z(G)\p_{u_\z}(F).
 $$
We define the Mayer bracket of $F_i,F_j$ by the formula
 $$
[F_i,F_j]_\E=\{F_i,F_j\}\,\op{mod}\mathcal{J}_{k_i+k_j-1}({\bf F}).
 $$
For first order scalar operators these are respectively the classical
Lagrange and Mayer brackets.

 \begin{theorem}[\cite{KL$_1$}]\label{MB}
Consider a scalar system $\E\subset J^k(M)$ given by $r\le n$ differential equations
$F_1[u]=0,\dots,F_r[u]=0$, such that for each point $x_k\in\E$ the
characteristic varieties for the equations $F_i$ are jointly
transversal, i.e. $\op{codim}\bigl[\op{Char}^\C_{x_k}(\E)\subset
\mathbb{P}\left( T_x^*M\right)^\C\bigr]=r$.

Then the system is formally integrable iff all the Mayer brackets
vanish:
 $$
[F_i,F_j]_\E=0,\ 1\le i<j\le r.
 $$
 \end{theorem}

The Koszul complex is the minimal resolution of the symbolic module $\mathcal{M}_\E$ for complete
intersections, whence the algebraic syzygies are generated by commutators.
Thus the arrow $\mathfrak{q}$ associates the higher Jacobi bracket $\{,\}$ to the commutator $[,]$
(and so can be treated as a quantization).

If the condition of complete intersection is violated, then the conclusion of the theorem could be wrong.
To see this let us re-visit the example from the Introduction.

%\smallskip

 \begin{examp}\rm
For system (\ref{NoC}) the compatibility conditions
$[F_1,F_2]_\E=0$, $[F_2,F_3]_\E=0$, $[F_3,F_1]_\E=0$ should necessarily hold
($F_i$ denotes the difference of the left and right hand sides in the $i$-th equation of (\ref{NoC})),
but they do not form a basis in differential syzygy module
(and so are not sufficient for compatibility).
A basis is given by relation (\ref{Dsy}) and a similar relation
 $$
\D_0(F_1)+\D_1(F_2)+\D_2(F_3)\equiv0\mod\E.
 $$
Bracket-relations are differential corollaries of these two, which yield a basis
in the second cohomology group (counting compatibility conditions) $H^{1,2}(\E)=\R^2$.
It is easy to see that the system $\E$ is not a complete intersection:
the characteristic variety is the normal cubic
 \begin{eqnarray*}
\op{Char}^\C(\E)&=&\{\xi\in\C P^3\,|\,\xi_1\xi_3=\xi_2^2,\xi_1\xi_2=\xi_0\xi_3,\xi_0\xi_2=\xi_1^2\}\\
&=&\{[\l^3:\l^2:\l:1]\,|\,\l\in\bar\C=\C\cup\infty\}
 \end{eqnarray*}
and it cannot be given by $2=\op{codim}\op{Char}^\C(\E)$ equations.
 \end{examp}

We can however have other explicit formulae for scalar non-complete intersections.
Let us point out one example. The following statement can be proved similarly to
Theorem \ref{MB} (and it follows from Theorem \ref{DS}).

 % Consider the system
 %  $$
 % \E=\{\Delta_1=0,\dots,\Delta_r=0\},
 %  $$
 % where $\Delta_i=Q\circ F_i+\dots$ and dots stay for the lower order terms.
 % Assume that the following is a complete intersection:
 %  $$
 % \Sigma^{n-r-1}=\{\z(\ell_{F_1})=0,\dots,\z(\ell_{F_r})=0\}.
 %  $$
 % In this case the characteristic variety
 % $\op{Char}^\C(\E)=\{\z(\ell_Q)=0\}\cup\Sigma$ is reducible and has a
 % component of codimension 1. The compatibility writes as
 %  $$
 % Q\circ\{F_i,F_j\}=0\,\op{mod}\mathcal{J}_{k_i+k_j+\op{ord}(Q)-1}({\bf\Delta}).
 %  $$

 \begin{theorem}
Consider the system
 $$
\E=\{F_i\circ G_j+\dots{}=0\,|\,i=1\dots r,j=1\dots s\}
 $$
where dots stand for the lower order terms.
It has reducible characteristic variety $\op{Char}^\C(\E)=\{\z(\ell_{F_i})=0\}\cup\{\z(\ell_{G_j})=0\}$,
and we suppose that the intersection $\{\z(\ell_{F_i})=0\}\cap\{\z(\ell_{G_j})=0\}$
has codimension $r+s$. % (is a complete intersection).
Then if $k_i=\op{ord}(F_i)$, $l_j=\op{ord}(G_j)$, the compatibility conditions are
 $$
\{F_i,F_j\}G_k=0\,\op{mod}\mathcal{J}_{k_i+k_j+l_k-1}(\E),
F_i\{G_j,G_k\}=0\,\op{mod}\mathcal{J}_{k_i+l_j+l_k-1}(\E).
 $$
 \end{theorem}

Let us now consider vector systems of PDEs on $u=(u^1,\dots,u^m)$.

 \begin{dfn}
A system $\E\subset J^k(\pi)$ of PDEs $F_1[u]=0,\dots,F_r[u]=0$ is called a
generalized complete intersection if
 \begin{enumerate}
 \item $m<r\leq n+m-1$, where $n=\dim M$, $m=\op{rank}(\pi)$;
 \item The complex projective characteristic variety
$\op{Char}^\C(\E)\subset\mathbb{P}^\C T^*$ has codimension
$r-m+1$;
 \item The characteristic sheaf $\mathcal{K}$ over
$\op{Char}^\C(\E)$ has fibers of dimension $1$ everywhere.
\end{enumerate}
 \end{dfn}

Let us introduce a multi-bracket of linear (scalar) differential operators
$\nabla_i\in\op{Diff}(m\cdot\1,\1)$ by the formula
 $$
\{\nabla_1,\dots,\nabla_{m+1}\}=\sum_{k=1}^{m+1}(-1)^{k-1}
\op{Ndet}\bigl[\nabla_i^j\bigr]_{i\ne k}^{1\le j\le m}\cdot\nabla_k,
 $$
where $\op{Ndet}$ is a version of non-commutative determinant.
For non-linear differential operators $F_i\in\op{diff}(\pi,\1)$ the multi-bracket is
 \begin{multline*}
\hspace{-0.1in}\{F_{1},\dots,F_{m+1}\}=\\
\dfrac{1}{m!}\hspace{-0.1in}%
\sum_{\alpha\in\mathbf{S}_{m},\beta\in\mathbf{S}_{m+1}}\hspace{-0.1in}\left(
-1\right)  ^{\alpha}\left(  -1\right)
^{\beta}\ell_{\alpha(1)}(F_{\beta \left(  1\right)
})\circ\ldots\circ\ell_{\alpha(m)}(F_{\beta\left(  m\right) })\left(
F_{\beta(m+1)}\right).
 \end{multline*}

Then we define the reduced bracket
 $$
[F_{i_1},\dots,F_{i_{m+1}}]_{\E}=\{F_{i_1},\dots,F_{i_{m+1}}\}\,
\op{mod}\mathcal{J}_{k_{i_1}+\cdots+k_{i_{m+1}}-1}({\bf F}).
 $$

 \begin{theorem}[\cite{KL$_4$}]\label{multibr}
A system of generalized complete intersection type
 $$
\E\subset J^k(\pi)=\{F_1[u]=0,\dots,F_r[u]=0\}
 $$
is formally integrable iff all the multi-brackets vanish due to the system:
 $$
[F_{i_1},\dots,F_{i_{m+1}}]_{\E}=0.
 $$
 \end{theorem}

For such systems $\E$ the map $\mathfrak{q}$ associates the multi-brackets of
differential operators to certain determinental identities, and these multi-brackets satisfy
the Pl\"ucker identity \cite{KL$_4$}.

%%%%%%%%%%%%%%%%%%%%%%%%%%%%%%%%%%%%%%%%%%%%%%%%%%%%%%%%%%%%%%%%%%%%%%%%%%%%
%2%
\section{Symmetries and compatibility}\label{S2}

Consider a compatible system $\E=\{F_1=0,\dots,F_r=0\}$. Let $\mathcal{G}$ be a subalgebra
of the algebra $\op{sym}(\E)$ of classical or higher symmetries \cite{KLV},
written as another PDE system $\{S_1=0,\dots,S_k=0\}$.

Consider at first the scalar case, when $u$ is a single function.
Denote $k_i=\op{ord}(F_i)$, $l_j=\op{ord}(S_j)$. Then the symmetry condition is
 \begin{equation}\label{4r}
\{F_i,S_j\}=0\,\op{mod}\mathcal{J}_{k_i+l_j-1}(\E).
 \end{equation}
In addition we have from the symmetry condition that
 \begin{equation}\label{r4}
\{S_i,S_j\}=0\,\op{mod}\mathcal{J}_{l_i+l_j-1}(\mathcal{G}).
 \end{equation}
 \begin{theorem}\label{tre}
Assume that the joint system $\E+\mathcal{G}$: $\{F_i=0,S_j=0\}$ is a complete intersection,
i.e. its characteristic variety has codimension $r+k\le n$. Then this system is compatible.
 \end{theorem}

For $k=1$ this result coincides with Theorem 17 from \cite{KL$_1$}.

 \begin{proof}
Since the system is a complete intersection, Theorem \ref{MB} implies compatibility provided
that the Mayer brackets (on the joint system $\E+\mathcal{G}$) vanish. Compatibility of $\E$ yields
 $$
\{F_i,F_j\}=0\,\op{mod}\mathcal{J}_{k_i+k_j-1}(\E).
 $$
Vanishing of the other brackets is given by (\ref{4r}) and (\ref{r4}).
 \end{proof}

The condition of complete intersection is important. %, as the following example shows.

%\smallskip

 \begin{examp}\rm
For the KdV equation $u_t=u\,u_x+u_{xxx}$ the algebra of classical symmetries is generated by:
$T_0=u_t$, $T_1=u_x$, $R=3tu_t+xu_x+2u$, $\Gamma=tu_x+1$. For every symmetry (one-dimensional
subalgebra) there are invariant solutions. Moreover the joint system KdV+\,symmetry is compatible, so the
space of invariant solutions had the expected dimension (1 for a symmetry $S\in\langle T_1,\Gamma\rangle$
and 3 for any other symmetry).

However there are no nontrivial invariant solutions for the following two-dimensional subalgebras of symmetries:
1) $\langle T_0,R\rangle$; 2) $\langle T_1,R\rangle$; 3) $\langle T_1,\Gamma\rangle$; 4) $\langle\Gamma,R\rangle$
(only zero solution for the joint system KdV+\,2\ sym\-met\-ries in the cases 1 \& 2 and nothing at all in 3 \& 4).
And indeed, in this case the joint system $\E+\mathcal{G}$ is not a complete intersection.

A similar observation on (non-)existence of invariant solutions applies to the infinite-dimensional
higher symmetry algebra  of the KdV.
 \end{examp}

Now let us discuss the general case, when the unknown $u$ is a vector-function.
Consider at first the following example of matrix linear differential operators:
 % $A=\begin{pmatrix} \D_x^2 & 0 \\ \D_x\D_y-\D_x & 0 \end{pmatrix}$ and
 % $B=\begin{pmatrix} \D_x\D_y-1 & 0 \\ \D_y^2-\D_y & \end{pmatrix}$.
 $$
A=\begin{pmatrix} \D_x\D_y & 0 \\ \D_x\D_y-\D_x & \D_x \end{pmatrix}\text{ and }
B=\begin{pmatrix} \D_x\D_y+\D_x & 0 \\ \D_x\D_y & \D_x \end{pmatrix}.
 $$
They commute $[A,B]=0$ and so do the inhomogeneous operators
 $$
F=A\cdot\begin{pmatrix} u \\ v\end{pmatrix}-\begin{pmatrix} 1 \\ 0\end{pmatrix},\ G=B\cdot\begin{pmatrix} u \\ v\end{pmatrix}\quad\Rightarrow\qquad \{F,G\}=0.
 $$
In other words $G\in\op{sym}(F)$. However the operators are incompatible.

We shall show that for two generic nonlinear differential operators the condition $G\in\op{sym}(F)$ implies compatibility of the joint system, and so existence of invariant solutions (in the amount given by the
usual formal calculus of dimensions \cite{C$_3$}).

This result was noticed by S.Igonin and A.Verbovetsky and a proof using a different method will appear in [IV].
We would like to give an explicit criterion for this compatibility. At first we notice the following
 \begin{lem}\label{Lem}
Consider a system $\E$ of $r\ge m$ differential equations for $m$ unknown functions.
Its characteristic variety $\op{Char}^\C(\E)$ has codimension $\le r-m+1$.
 \end{lem}

Notice that codimension cannot exceed $n$ since we adopt the convention $\dim\emptyset=-1$
and the ambient space satisfies $\dim\mathbb{P}^\C T^*=n-1$.

 \begin{proof}
The symbol of the system is a $r\times m$ matrix $P$ with polynomial entries, and the characteristic
variety is
 $$
\{\xi\in \mathbb{P}^\C T^*:\op{rank}(P(\xi))<m\}.
 $$
Denote by $\Delta_{i_1\dots i_m}$ the determinant of the $m\times m$ minor generated by rows $i_1,\dots,i_m$.
$\op{Char}^\C(\E)$ is given by the conditions $\Delta_{i_1\dots i_m}(\xi)=0$ for all
$i_1<\dots<i_m$, but this collection is excessive.

We can suppose that $\Delta_{1\dots m}\not\equiv0$ (if all the minors are degenerate there is nothing to prove).
For $j\ge m$ let us denote by $P_j$ the upper-row submatrix of $P$ of size $j\times m$.
By induction we can suppose that the set
$\Sigma_j=\{\xi\in \mathbb{P}^\C T^*:\op{rank}(P_j(\xi))<m\}$ has codimension $\le j-m+1$.
Consider the subset $\Sigma_j'$ where the above rank is $<m-1$.

If $\Sigma_j'$ has a component in $\Sigma_j$, then addition of a row to $P_j$ does not increase the rank over $m-1$,
and the codimension of $\Sigma_{j+1}$ containing this component is the same as for $\Sigma_j$.

Otherwise $\Sigma'_j\subset\Sigma_j$ has codimension at least 1 and in the complement the matrix $P_j(\xi)$
has rank $(m-1)$. Near every point $\xi$ we can choose $(m-1)\times m$ subminor of maximal rank.
Adding to it the row number $(j+1)$ we get a $m\times m$ matrix whose determinant we write as $\tilde\Delta_{j+1}$.
Then the defining relation for $\Sigma_{j+1}\subset\Sigma_j$ outside $\Sigma_j'$ is $\tilde\Delta_{j+1}=0$,
whence the relative codimension is 1.

Thus codimension of $\Sigma_{j+1}$ in the complex projective variety is at most $j-m$ and this gives the
induction step.
 \end{proof}

In particular, for $r=2m$ the codimension is at most (and generically for $m<n$ it is) $m+1$.
We shall refine this in the case, when the system $\E$ is a PDE coupled with its symmetry.

\smallskip

Consider the PDE system $F=G=0$, where $F,G\in\op{diff}(\pi,\pi)$ are nonlinear differential
operators from a rank $n$ bundle $\pi$ to itself and $G\in\op{sym}(F)$.
Denote the symbols of these operators by $P=\z(\ell_F)$, $Q=\z(\ell_G)$.
The symmetry condition $\{F,G\}=0\,\op{mod}F$ implies the following relation on
polynomial matrices $P,Q$ with some other polynomial matrix $K$
 \begin{equation}\label{comN}
PQ=KP.
 \end{equation}

 \begin{prop}\label{Pro}
For $m>1$ let $P$ and $Q$ be two homogeneous polynomial $m\times m$ matrices satisfying (\ref{comN}).
Then the characteristic variety
 $$
\op{Char}^\C=\left\{\xi\in\C P^{n-1}:\op{rank}\Bigl[{{P(\xi)}\atop{Q(\xi)}}\Bigr]<m\right\}
 $$
(characteristic variety of $\E:\{F=G=0\}$) has codimension $\le m$.
 \end{prop}

 \begin{proof}
Denoting by $R_i(A)$ the row of matrix $A$, we have from the symmetry condition:
 \begin{equation}\label{poi}
\sum_{k=1}^m p_{ik}R_k(Q)\in\langle R_1(P),\dots,R_m(P)\rangle\quad \forall i=1,\dots,m.
 \end{equation}
We can suppose that $\det P(\xi)\not\equiv0$ (otherwise the claim follows from the Lemma), so that
the equation $\det P(\xi)=0$ determines a subvariety $\Sigma\subset \mathbb{P}^\C T^*$
of codimension 1.

Let $\Sigma_0\subset\Sigma$ be given by the equation $P(\xi)=0$.
First let us study the points $\xi\in\Sigma\setminus\Sigma_0$.
At such $\xi$ there exists an entry $p_{ik}\ne0$. Then we write the conditions
$R_j(Q)\in\langle R_1(P),\dots,R_m(P)\rangle$ for $j\ne k$. These are no more than
$m-1$ equations and so they specify a subvariety $\mathcal{K}\subset\Sigma$ in a neighborhood of $\xi$
on which also $R_k(Q)\in\langle R_1(P),\dots,R_m(P)\rangle$ by (\ref{poi}).
Thus this $\mathcal{K}$ is a part of $\op{Char}^\C$ of codimension $\le m$.

Now let us consider $\Sigma_0$. If it has codimension $\ge m$, it is negligible
or is a part of $\op{Char}^\C$ by the above argument.
But if its codimension is $<m$, then we must add the condition $\det Q(\xi)=0$ specifying
a subvariety $\mathcal{K}_0\subset\Sigma_0$ of codimension $\le1$. Since $P(\xi)=0$
for $\xi\in\Sigma_0$ the rank of the matrix $\Bigl[{{P(\xi)}\atop{Q(\xi)}}\Bigr]$ is $<m$.
Thus this $\mathcal{K}_0$ is a part of $\op{Char}^\C$ of codimension $\le m$.
 \end{proof}

This Proposition is important, so we would like to indicate an idea behind an alternative proof.
It will be shown later that condition (\ref{comN}) can be changed
to $[P,Q]=0$ without loss of generality, so we adopt this condition.

By Gerstenhaber theorem \cite{Ge} every pair of numeric (complex) commuting matrices $P,Q$
is contained in a commutative algebra (with $\1$) of dimension $m$.
If one of them has simple spectrum, say $P$, then by Cayley-Hamilton theorem this algebra
is generated by $\{P^i\}_{i=0}^{m-1}$ \cite{Z}.

This being generalized to matrices with polynomial entries, would imply
existence of polynomial matrices
$Z_0,Z_1,\dots,Z_{m-1}$ and the scalar polynomials $a_i,b_j$ such that
 $$
P=a_0Z_0+a_1Z_1+\dots+a_{m-1}Z_{m-1},\ \ Q=b_0Z_0+b_1Z_1+\dots+b_{m-1}Z_{m-1}.
 $$
Therefore one of the strata of the characteristic variety is given by the condition that the vector $(a_0,\dots,a_{m-1})$ is parallel to $(b_0,\dots,b_{m-1})$ ($m-1$ equations) plus one equation
$\det P=0$ (or $\det Q=0$), implying $\op{codim}\op{Char}^\C\le m$.

Let us demonstrate this in the case $m=2$.
Since identity matrices commute with everything, we can
subtract a multiple of them to make $p_{22}=q_{22}=0$. Then commutativity $[P,Q]=0$ yields:
 $$
\Bigl[\begin{pmatrix} p_{11} & p_{12} \\ p_{21} & 0 \end{pmatrix},
\begin{pmatrix} q_{11} & q_{12} \\ q_{21} & 0 \end{pmatrix}\Bigr]=0\quad
\Rightarrow\ \frac{p_{11}}{q_{11}}=\frac{p_{12}}{q_{12}}=\frac{p_{21}}{q_{21}},
 $$
which means that these matrices are proportional to some polynomial matrix.
Indeed, the first equality gives $p_{11}=r_1s_1$, $q_{11}=r_1s_2$, $p_{12}=r_2s_1$, $q_{12}=r_2s_2$
for some polynomials $r_i,s_i$.

Substituting this into the second equality we get
refining: $s_1=u_1v_1$, $s_2=u_1v_2$, $p_{21}=u_2v_1$, $q_{21}=u_2v_2$,
whence the claim $P=a_0\1+a_1Z$, $Q=b_0\1+b_1Z$.

\medskip

Let us return to sufficient conditions for compatibility of symmetries.

From what was shown above it follows that in the non-scalar case $\E=\{F=G=0\}$ is
never a generalized complete intersection (so compatibility cannot be deduced on the basis
of Theorem \ref{multibr}).
This is a consequence of non-commutativity of the matrix algebra;
another feature of the algebra of matrix differential operators ($m>1$) is that generically
$\op{ord}\{F,G\}=\op{ord}(F)+\op{ord}(G)$.

 \begin{theorem}\label{thM8}
Let $F=0$ be a determined PDE and $G=0$ its symmetry (both $m\times m$ systems).
If for the joint system $\E$: $F=G=0$ the variety $\op{Char}^\C(\E)$ has codimension $m$, then
$\E$ is compatible.
 \end{theorem}

The idea is that since this is the maximal possible codimension, there is no space for another
syzygy except for the symmetry relation.

 \begin{proof}
Let us first give the proof in a particular case, when both $F$ and $G$ are of the first order.
Since the equation $F=0$ is determined ($\det P$ is not identically zero),
we can write it in the evolutionary form.

Thus $F[u]=u_t-F_0[u]$, where $F_0$ does not involve $u_t$ terms but contains some other derivatives of $u$
($u$ is a vector function with $m$ components). We can substitute $u_t=F_0$ into $G=0$ and get a PDE
that is free of $u_t$ terms. We continue to write $G$ for this new operator.

Denoting the symbols of $F$ and $G$ by $P$ and $Q$, the symmetry condition is $PQ=KP$ for some
polynomial matrix $K$. Since $P=\t\1+P_0$, where $\t$ stands for the symbol of $\D_t$ differentiation
and $P_0$ (as well as $Q$) is free of $\t$, we conclude that $K=Q$, i.e. $[P,Q]=0$.

We claim that $G$ is also determined, i.e. $\det Q\not\equiv0$.
Otherwise the characteristic equation is given by $m-1$ equations at generic point.
Indeed, the proof of Lemma/Proposition \ref{Pro}
can be rewritten to start with equation $\det Q=0$,
which is identically true and so the number of defining relations for the variety
$\{\xi:\op{rank}\Bigl[{{P}(\xi)\atop{Q}(\xi)}\Bigr]<m\}$ is by 1 less than in the general case.
This yields $\op{codim}\op{Char}^\C(\E)<m$ and so contradicts the assumptions.

Thus the system $G=0$ is compatible (as well as the system $F=0$). Due to evolutionary form
of $P$ (and absence of $\t$ in $Q$) there are no syzygies of the system $P=Q=0$ besides the
commutation relation. So the claim follows from Theorem \ref{DS}.

Now let us consider the general case, when the orders can exceed 1. Let us re-write the system
in the 1st order evolutionary form. We keep the same letters $F$ and $G$ for the differential operators.
This invertible transformation does not affect the symmetry or compatibility properties and it preserves $\op{Char}^\C(\E)$, but it increases $m$.

This increase is composed of introduction of new variables
and re-writing the operators $F$ and $G$ via them. The first type equations (including equalities of
mixed derivatives) have no influence on the characteristic variety and when coupled with new $F$ they give
the same codimension 1 sub-variety in $\C P^{n-1}$ as the old $F$. The new equation $G$ (re-written
through the new variables and not involving $t$-derivatives) still consists of $m$ equations, and by
the same argument as above it is determined. The symmetry condition written as $\D_t(G)|_{u_t=F_0[u]}=0$
(now $u$ consists of old and new dependent functions) is the only differential syzygy,
and we conclude that $\E=\{F=G=0\}$ is compatible.
 \end{proof}

%\smallskip

 \begin{examp}\rm
Consider the Kadomtsev-Pogutse equation (which is a simplification of the MHD system):
 \begin{eqnarray*}
F_1&=&\psi_t+\vp_x\psi_y-\vp_y\psi_x-\vp_z=0,\\
F_2&=&\vp_{xxt}+\vp_{yyt}+\vp_x\vp_{xxy}+\vp_x\vp_{yyy}-\vp_y\vp_{xxx}-\vp_y\vp_{xyy}-\\
   &&-\psi_{xxz}-\psi_{yyz}-\psi_x\psi_{xxy}-\psi_x\psi_{yyy}+\psi_y\psi_{xxx}+\psi_y\psi_{xyy}=0.
 \end{eqnarray*}
Its symmetries consist of 7 functional families and 2 additional fields, see \cite{KV}.
For instance here is the first family:
 \begin{eqnarray*}
\mathcal{A}_1(\a)&=&\a'\cdot(\vp+\psi)+\a\cdot(\vp_z+\vp_t),\\
\mathcal{A}_2(\a)&=&\a'\cdot(\vp+\psi)+\a\cdot(\psi_z+\psi_t),
 \end{eqnarray*}
where $\a=\a(z+t)$. The characteristic variety of the system $\E=\{F=\mathcal{A}=0\}$
has codimension $m=2$,
 \begin{multline*}
\op{Char}^\C(\E)=\{[p_x:p_y:p_z:p_t]\in\C P^3\,:\,p_z+p_t=0\,\&\\
[\,p_x^2+p_y^2=0\vee (p_z+\psi_x p_y-\psi_y p_x)^2-(p_t+\vp_x p_y-\vp_y p_x)^2=0\,]\}.
 \end{multline*}
Thus by Theorem \ref{thM8} the system $\E$ is compatible, and so gives rise to
invariant solutions (see \cite{KV} and references therein for the exact solutions of the
Kadomtsev-Pogutse equation arising via symmetry methods).
 \end{examp}

%\smallskip

As we see from the above proof it is not necessary to require that $\op{codim}\op{Char}^\C(\E)=m$.
What we want to achieve is that the operator $G$, when $F$ is written in evolutionary
form and $\D_t$ derivatives are removed from $G$, is determined. This is clearly a
generic property.

Let us now give a sufficient condition in a special case
when the symbol $P=\z(F)$ ($m\times m$ matrix with polynomial
entries) can be diagonalized via an invertible transformation over polynomials.
As will be seen from the proof, $\op{codim}\op{Char}^\C(\E)=2$
and so this case does not follow from the above theorem (if $m>2$).

In the proof we will calculate the $\op{Mat}_{m\times m}(ST)$-module
 $$
\op{Syz}(P,Q)=\{(A,B)\in\op{Mat}_{m\times 2m}(ST):AP+BQ=0\}.
 $$
It encodes all syzygies which are the symbols of differential syzygies, and the latter
are the compatibility conditions by Theorem \ref{DS}.

 \begin{theorem}
Assume that $P=\z(\ell_F)$ is diagonalizable over the algebra $ST$. Assume also that
the characteristic variety of $F=0$ has codimension 1 (i.e. $F$ is determined)
and has no multiple components.

Let $G\in\op{sym}(F)$ be a symmetry (both $F,G$ are $m\times m$ PDE systems). If
for the joint system $\E:F=G=0$ the characteristic variety has codimension $>1$, then $\E$ is compatible.
 \end{theorem}

 % we will not use reduction to the evolutionary form in the proof!

 \begin{proof}
We can assume
from the beginning that $P$ is diagonal $P=\op{diag}(p_1,\dots,p_n)$ for some polynomials on $T^*$: $p_i\in ST$.
Denote the symbol of the symmetry $G$ by $Q$. This is also a matrix with polynomial entries,
$Q=[q_{ij}]_{m\times m}$.

The symmetry condition gives the following matrix syzygy:
 $$
PQ=KP\ \Leftrightarrow q_{ij}p_i=k_{ij}p_j \text{ (no summation).}
 $$
Because the characteristic variety has no multiple components,
we conclude from here $q_{ij}=r_{ij}p_j$, $k_{ij}=r_{ij}p_i$ for $i\ne j$.
This means that $Q=RP+D$ and $K=PR+D$ for some diagonal matrix $D$. Since $\op{codim}\op{Char}^\C(\E)\ge2$,
 % the matrix $D$ is not divisible by $P$, more precisely
$p_i$ and $q_{ii}$ have no common factors.

Then the syzygy $AP+BQ=0$ is $(A+BR)P+BD=0$, which yields $B=-LP$ and $A=LD-BR=LK$.
Consequently the $\op{Mat}_{m\times m}(ST)$-module $\op{Syz}(P,Q)$ is 1-dimensional
and generated by $(K,-P)$. The result follows from Theorem \ref{DS}.
  \end{proof}

 \begin{rk}
Based on the generators of the syzygy module $\op{Syz}(P,Q)$ one can check that multi-brackets
of the rows of the $2m\times m$ matrix differential operator $\Bigl[{{F}\atop{G}}\Bigr]$
vanish as a consequence of the symmetry condition.
 \end{rk}

A similar argument shows that an analog of Theorem \ref{tre} holds for vector nonlinear differential
operators, under some genericity assumptions on the characteristic variety/ideal (implying that
dimension $n$ shall be sufficiently big compared to the number of equations in $\E$ and $m>1$).

%\smallskip

 \begin{examp}\rm
Consider the non-linear wave equation
 $$
F[v]^i=v^i_{tt}-\sum_{j=1}^n\frac{\p}{\p x^j}\Bigl(|v|^{-2a}\frac{\p v^i}{\p x^j}\Bigr)=0,
 $$
where $v^i\in C^\infty(\R^{1+n}(t,x^1,\dots,x^n))$, $i=1,\dots,m$.
Its symbol is diagonal, and the scaling symmetry
$G[v]^k=\xi(v^k)$ for
 $$
\xi=t\p_t+(1-2a)\sum_{j=1}^n x^j\p_{x^j}+2\sum_{i=1}^m v^i\p_{v^i}
 $$
satisfies the property that the characteristic variety of the joint system $\E:F=G=0$ has codimension 2.
This system $\E$ is compatible and there are invariant solutions.
 \end{examp}

%%%%%%%%%%%%%%%%%%%%%%%%%%%%%%%%%%%%%%%%%%%%%%%%%%%%%%%%%%%%%%%%%%%%%%%%%%%%
%3%
\section{Exact solvability of differential equations}\label{S3}

According to Cartan's test solutions to a compatible system of differential equations $\E$
are formally parametrized by $p$ functions of $g$ arguments (and some functions with fewer arguments).
These numbers are expressed through the characteristic variety and the characteristic sheaf $\mathcal{K}$
over it as follows: $g=\dim_\C\op{Char}^\C(\E)+1$, $p=\op{deg}\op{Char}^\C(\E)\cdot\dim\mathcal{K}$,
see \cite{BCG$^3$,KL$_2$}. Denote the abstract space of such functions by $\mathfrak{S}_p^g$.

Closed form solutions refer to parametrization of the generic stratum
of the solutions space by a differential operator $S:\mathfrak{S}_p^g\stackrel\sim\to\op{Sol}(\E)$.
By \cite{C$_2$} for underdetermined ODEs ($g=1$) this is
tantamount\footnote{In loc.cited only the classical Monge case $p=1$ (1 equation on 2 unknowns) was treated,
but the general case is similar.} to
internal equivalence of the equation equipped with the contact distribution $(\E,\mathcal{C}_\E)$
to some mixed jet space $J^\z(\R,\R^p)$. %equipped with the canonical distribution.
Here $\z=(i_1,\dots,i_p)$ is a multi-index characterizing
the orders of the dependent variables.

%\smallskip

 \begin{examp}\rm
Consider the equation for null curves in Lorentzian space of $1+2$ dimensions
 $$
x'(t)^2+y'(t)^2=1\quad\Leftrightarrow\quad dt^2-dx^2-dy^2=0.
 $$
  % (in $1+1$ dimensions the null curves are pre-geodesic and are characteristics of the eikonal equation)
An obvious solution involves one arbitrary function and the quadrature:
$x=\int\cos\phi(t)dt$, $y=\int\sin\phi(t)dt$.
But it can also be integrated in the closed form:
 \begin{multline*}
t=\sigma''(\tau)-\sigma(\tau),\quad x=\sigma''(\tau)\cos\tau+\sigma'(\tau)\sin\tau,\\ y=-\sigma''(\tau)\sin\tau+\sigma'(\tau)\cos\tau.\qquad
 \end{multline*}

This equation in $1+n$ dimensions cannot be solved in closed form for $n>2$. Indeed internally
 $$
\E=\R^{n+1}\times S^{n-1}=\{(t,x_1,\dots,x_n,x_1',\dots,x_n'):\sum_{i=1}^n(x_i')^2=1\}.
 $$
The Cartan distribution on $\E$ is generated by $\D_t=\p_t+\sum_1^nx_i'\p_{x_i}$ and the subspace
$\Pi=TS^{n-1}=\langle\p_{x_1'},\dots,\p_{x_n'}|\sum x_i'\p_{x_i'}=0\rangle$: $\mathcal{C}_\E=\R\D_t+\Pi$.

Let $\Delta_1=\mathcal{C}_\E$ and $\Delta_2=[\Delta_1,\Delta_1]$ be its derived distribution.
Then $\dim\Delta_1=n$, $\dim\Delta_2=2n-1$ and we have
$\Delta_2=\Delta_1+\langle x_i'\p_{x_j}-x_j'\p_{x_i}\rangle$. The next derived
distribution is $\Delta_3=[\Delta_1,\Delta_2]=T\E$.

Thus by dimensional reasons the only possible corresponding jet-space is
$J^{\z}(\R,\R^{n-1})$, $\z=(1,\dots,1,2)$, with coordinates
$t,u_1,\dots,u_{n-1}$, $u_1',\dots,u_{n-1}',u_{n-1}''$. Its Cartan distribution is
 $$
\tilde\Delta_1=\langle\D_t,\p_{u_1'},\dots,\p_{u_{n-2}'},\p_{u_{n-1}''}\rangle,\
\text{ where }\
\D_t=\p_t+\sum_{i=1}^{n-1}u_i'\p_{u_i}+u_{n-1}''\p_{u_{n-1}'}.
 $$
The derived distribution is equal to $\tilde\Delta_2=\tilde\Delta_1+\langle\p_{u_1},\dots,\p_{u_{n-2}},\p_{u_{n-1}'}\rangle$
and $\tilde\Delta_3=TJ^\z$.

Though the dimensions coincide, the distributions on $\E$ and $J^\z$ are not equivalent. The reason is that
$\Pi$, which is the maximal involutive space of the bracket $\La^2\Delta_1\to\Delta_2/\Delta_1$, is not the Cauchy
characteristic space for $\Delta_2$. But in the second case $\tilde\Pi=\langle\p_{u_1'},\dots,\p_{u_{n-2}'},\p_{u_{n-1}''}\rangle$,
which is the maximal involutive space of the bracket $\La^2\tilde\Delta_1\to\tilde\Delta_2/\tilde\Delta_1$, is the Cauchy
characteristic space for $\tilde\Delta_2$.
 \end{examp}

%\smallskip

 \begin{examp}\rm
Consider the Monge equation $w'(x)=(z'(x))^2$. The general solution depends on 1 function
of 1 variable and the form via quadrature is obvious, but here is the closed form solution:
 $$
x=\sigma''(\tau),\ w=\tau^2\sigma''(\tau)-2\tau\sigma'(\tau)+2\sigma(\tau),\ z=\tau\sigma''(\tau)-\sigma'(\tau).
 $$
The explanation behind this is the Engel normal form for rank 2 distributions in $\R^4$.
However the next candidate -- the Hilbert-Cartan equation
 \begin{equation}\label{HC}
w'=(z'')^2
 \end{equation}
is no longer integrable in closed form (without quadratures) as was demonstrated in 1912 by D.Hilbert.
In 1914 \'E.\,Cartan gave a criterion for resolution of underdetermined ODEs in closed form \cite{C$_2$},
which we referred to above.

Already in 1910 Cartan found that the symmetry group of (\ref{HC}) is $G_2$
(though it was not written like this in \cite{C$_1$}, he surely knew this) and proved that it is
the most symmetric equation among such Monge equations with finite-dimensional symmetry groups
(linearizable equations have infinite-dimensional group of symmetries).
 \end{examp}

%\smallskip

The more general problem when a PDE is integrable in closed form via solutions of a
simpler equation (usually ODE) is known as the method of Darboux (that is we allow
for the operator $S$ above to involve quadratures and some other nonlocalities).
For instance, the Liouville equation is Darboux integrable
 $$
u_{xy}=e^u\quad\Longrightarrow\quad u=\log\frac{2f'(x)g'(y)}{(f(x)+g(y))^2},
 $$
while the sin-Gordon equation $u_{xy}=\sin u$ is not.

 \begin{examp}\rm
The following overdetermined system of PDE on $\R^2$ appeared in \cite{C$_1$}
($\l$ is a parameter to be excluded):
 \begin{equation}\label{dagger}
u_{xx}=\tfrac13\lambda^3,\ u_{xy}=\tfrac12\lambda^2,\ u_{yy}=\lambda.
 \end{equation}
It is a compatible involutive system.
The general solution is parametrized by 1 function of 1 argument
 \begin{multline*}
x=x,\ y=z''(\t)+x\t,\\ u=xz(\t)+z'(\t)z''(\t)-\tfrac12w(\t)-\tfrac12\t\,z''(\t)^2-
\tfrac12\t^2x\,z''(\t)-\tfrac16\t^3x^2,
 \end{multline*}
where $w'(\t)=(z''(\t))^2$.
In fact, this system has a common characteristic $\p_x-\l\p_y$, which lifts to the Cauchy characteristic
of the Cartan distribution (of rank 3). The quotient by the Cauchy characteristic is a 5-dimensional
manifold with rank 2 distribution equivalent to (\ref{HC}).

By a Lie-B\"acklund type theorem \cite{C$_1$} the contact symmetry group of (\ref{dagger})
coincides with the internal symmetry group of (\ref{HC}), and so is $G_2$.

According to Goursat \cite{Gou} the general form of overdetermined involutive (in this case: compatible
with a common characteristic) system of 2nd order PDE on the plane is
 \begin{equation}\label{GtE}
r+2\lambda s+\lambda^2t=2\psi,\  s+\lambda t=\psi_\lambda,\ t=\psi_{\lambda\lambda},
 \end{equation}
where we use the classical notations $r=u_{xx}$, $s=u_{xy}$, $t=u_{yy}$ and
suppose $\psi_{\lambda\lambda\lambda}\ne0$ (nonlinearity).
System (\ref{dagger}) corresponds\footnote{One also need to change $y\mapsto-y$ to match the sign.
\label{ftnt5}} to $\psi=\lambda^3/3!$

Removing the last equation from this system we obtain a determined parabolic PDE $\E$ of the 2nd order.
It has the largest contact symmetry group (among non-linear equations) for $\psi=\lambda^3/3!$
in which case excluding $\lambda$ we obtain
the Goursat equation
 \begin{equation}\label{ddagger}
4(2s-t^2)^3+(3r-6st+2t^3)^2=0.
 \end{equation}
This equation has the same symmetry group $G_2$ and it
can be parame\-trized as the 2D tangent cone $\rho(\lambda)+\mu\,\rho'(\lambda)$
to the twisted cubic $\rho(\lambda)$:
 $$
r=\tfrac13\lambda^3+\lambda^2\mu,\ s=\tfrac12\lambda^2+\lambda\mu,\ t=\lambda+\mu.
 $$
Excluding $\lambda$ and $\mu$ we get equation (\ref{ddagger}).

It has the following geometric reduction to (\ref{HC}), giving exact solutions for (\ref{ddagger}).
The double characteristic $\p_x-\lambda\p_y$ lifts to
 $$
\xi=\D_x-\lambda\D_y=\p_x-\lambda\p_y+(p-\l q)\p_u-\tfrac16\l^3\p_p-\tfrac12\l^2\p_q,
 $$
which together with $\p_\mu=\l^2\p_r+\l\p_s+\p_t$ forms the integrable characteristic rank 2
distribution $\Pi\subset\mathcal{C}_\E$. Quotient by it maps the Cartan distribution $\mathcal{C}_\E$
(of rank 4) to a rank 3 distribution $\bar\Delta$ on a 5-dimensional manifold $M^5$. This rank 3 distribution
is the derivative distribution of the unique rank 2 distribution $\Delta$ which maps to zero
under restriction of the natural bracket $\La^2\bar\Delta\to TM/\bar\Delta$.

If we identify $M^5$ with $\R^5(y,u,p,q,\l)$, then
$\Delta=\langle\p_y+q\p_u+\frac12\l^2\p_p+\l\p_q,\p_\l\rangle$ and
$\bar\Delta=\bar\Delta+\langle\p_q+\l\p_p\rangle$. Thus we see that $(M^5,\Delta)$
is equivalent to Hilbert-Cartan equation (\ref{HC}).
 \end{examp}

%\smallskip

The following deformation $\E_\epsilon$ of the Goursat equation was studied in \cite{T}:
 \begin{equation}\label{agger}
(4+\epsilon)(2s-t^2)^3+(3r-6st+2t^3)^2=0.
 \end{equation}
Here $\epsilon>0$ and for every such number (\ref{agger}) is hyperbolic\footnote{Outside
the submanifold in $\E_\epsilon$ given by Cartan equation (\ref{dagger})!}
and it has maximal symmetry algebra of dimension 9 among all hyperbolic 2nd order
PDEs on the plane, which are neither of Monge-Ampere nor of Goursat type (see \cite{T} for details).
The Lie algebra of symmetries has Levi decomposition $\mathfrak{g}=\op{sl}_2\ltimes\mathfrak{r}$
(radical $\mathfrak{r}$ is 6-dimensional).

Except for this family there is one more hyperbolic equation with 9-dimensional symmetry algebra
($\mathfrak{g}$ has the same form but different $\mathfrak{r}$)
 \begin{equation}\label{Ge1}
3rt^3+1=0.
 \end{equation}
All these equations are Darboux integrable. The last one was studied in Goursat \cite{Gou}.
Here is one intermediate integral of order 2:
 \begin{equation}\label{Ge2}
s\,t+1=0.
 \end{equation}
The system (\ref{Ge1})+(\ref{Ge2}) is involutive and so allows reduction to a rank 2 distribution.
Indeed this system can be re-written in the form similar to (\ref{dagger}) ---
using Goursat representation (\ref{GtE}) with $\psi=-\frac43\l^{3/2}$ we get$^\text{\ref{ftnt5}}$
 \begin{equation}\label{ECeq1}
u_{xx}=\tfrac13\l^{3/2},\ u_{xy}=\l^{1/2},\ u_{yy}=-\l^{-1/2}.
 \end{equation}
This system is compatible (compatibility writes formally as $\l_x=\l\l_y$
but this relation expresses the prolongation) and the characteristic field is $\xi=\D_x-\l\D_y$,
which is also the Cauchy characteristic field for the Cartan distribution $\mathcal{C}_\E$
on the equation \cite{K$_4$}.

Quotient of this rank 3 distribution $\mathcal{C}_\E$ by $\xi$ gives a rank 2 distribution on 5-manifold
$M^5$, which corresponds to the Monge equation
 \begin{equation}\label{1/3}
w'=(z'')^{1/3}.
 \end{equation}
This latter is equivalent to (\ref{HC}) and has $\op{Lie}(G_2)$ algebra of symmetries.
Thus by Cartan's version of Lie-B\"acklund theorem \cite{C$_1$} the contact symmetry algebra of
the involutive system (\ref{Ge1})+(\ref{Ge2}) is the same as for (\ref{1/3}) --- it's one
more representation of $G_2$.

 \begin{rk}
These all are realizations of the non-compact (split) form of $G_2$ as the maximal symmetric model.
The compact form of $G_2$ is realized by the automorphism group of the Calabi almost complex structure
$(S^6,J)$. By \cite{K$_5$} this is the maximal symmetric model that acts on
non-linear overdetermined non-integrable\footnote{This means the corresponding Nijenhuis tensor $N_J$ is non-degenerate.} Cauchy-Riemann equations
 $$
\zeta_{\bar z}=\Phi(z,w,\zeta),\ \zeta_{\bar w}=\Psi(z,w,\zeta).
 $$
 \end{rk}

Higher Monge equations were studied in \cite{AK}, and all maximal symmetric models were identified
as the following underdetermined ODEs: $y^{(m)}(x)=(z^{(n)}(x))^2$ (these again cannot be solved
without quadratures according to Cartan's criterion).

The symmetries here can be thought of as internal or external -- they coincide
by a version of Lie-Backlund theorem from \cite{AK}.

There are also PDE models for these according to \cite{K$_4$}. We will demonstrate some in
examples, which also indicate a relation to the projective geometry of curves
(the tool from \cite{DZ}) --- in our case (most symmetric models) these are
the rational normal curves.

%\smallskip

 \begin{examp}\rm
The following system is involutive\footnote{It has Lie class $\oo=1$, i.e.
the solutions depend upon 1 function of 1 argument.}
 \begin{equation}\label{relaxation}
u_{xxx}=\tfrac14\lambda^4,u_{xxy}=\tfrac13\lambda^3,u_{xyy}=\tfrac12\lambda^2,u_{yyy}=\lambda
 \end{equation}
It has type $3E_3$ in notations of \cite{K$_3$}. Quotient by Cauchy characteristic $\xi=\D_x-\l\D_y$
yields a rank 2 distribution on a manifold of dimension 8. The weak growth vector of this distribution
(we refer for the definition and properties to \cite{AK}) is $(2,1,2,3)$ and the
corresponding Monge system is
 \begin{equation}\label{monge11}
y''=\tfrac12(z''')^2,\ w'=\tfrac13(z''')^3.
 \end{equation}
The corresponding graded nilpotent (Carnot-Tanaka) Lie algebra is free truncated of length 4 with 2-dimensional
fundamental space $\mathfrak{g}_{-1}$.

Similarly we reduce $4E_4$
 $$
u_{xxxx}=\tfrac15\lambda^5,u_{xxxy}=\tfrac14\lambda^4,
u_{xxyy}=\tfrac13\lambda^3,u_{xyyy}=\tfrac12\lambda^2,u_{yyyy}=\lambda
 $$
to a rank 2 distribution with growth $(2,1,2,3,4)$ and the Monge equation
 $$
y'''=\tfrac12(z^{iv})^2,\ v''=\tfrac13(z^{iv})^3,\ w'=\tfrac14(z^{iv})^4.
 $$

We can modify the symmetric models without destroying the symmetry algebra.
For the above $3E_3$ we get its tangent 2D cone
 \begin{multline*}
\qquad\qquad u_{xxx}=\tfrac14\lambda^4+\lambda^3\mu,\ u_{xxy}=\tfrac13\lambda^3+\lambda^2\mu,\\
u_{xyy}=\tfrac12\lambda^2+\lambda\mu,\ u_{yyy}=\lambda+\mu.\qquad\qquad
 \end{multline*}
This type $2E_3$ system is compatible (the prolongation obeys $\l_x=\l\l_y$),
and its general solution depends on $\oo=2$ functions of 1 argument.

The characteristic is still $\D_x-\lambda\D_y$ (with multiplicity two) and
the contact symmetry algebra is the same as the contact external algebra for (\ref{relaxation})
and the same as the internal algebra for (\ref{monge11}): it has dimension 12 and
Levi decomposition $\mathfrak{g}=\op{sl}_2\ltimes\mathfrak{r}$
(radical $\mathfrak{r}$ is 9-dimensional and it consists of 1-dimensional center and
8-dimensional nil-radical).

Further modification gives the 3D tangent cone of the normal curve, which is a strictly
parabolic 3rd order determined PDE with the same 12-dimensional symmetry algebra $\mathfrak{g}$.
Denoting by $\a,\b,\g,\d$ the 3rd derivatives $u_{xxx},u_{xxy},u_{xyy},u_{yyy}$ we can write this
system of equations parametrically
 \begin{align*}
&& \a=\tfrac14\l^4+\l^3\mu+3\l^2\nu, \quad & \g=\tfrac12\l^2+\l\,\mu+\nu, \qquad {}\\
&& \b=\tfrac13\l^3+\l^2\mu+2\,\l\,\nu, \quad &  \d=\l+\mu, \qquad {}
 \end{align*}
or in the implicit form
 \begin{multline*}
\!\!\! 8\a^3-18\a^2 (4\b\d+8\g^2-12\g\d^2+3\d^4) +27\a\b^2 (18\g-\d^2)+4\a\b\g\d(3\d^2-10\g)\ \\
+8\a\g^3(3\g-\d^2) +27\b^2 (27\b\g\d-8\b\d^3-18\g^3+6\g^2\d^2) = 2187\,\b^4/8.
 \end{multline*}
Its solution space depends on $\oo=3$ functions of 1 argument and is
intrinsically related to the Monge system (\ref{monge11}).

Higher analogs of the above are valid on the basis of works \cite{AK,K$_4$}.
 \end{examp}

%%%%%%%%%%%%%%%%%%%%%%%%%%%%%%%%%%%%%%%%%%%%%%%%%%%%%%%%%%%%%%%%%%%%%%%%%%%%
%4%
\section{Conclusion}\label{S4}

In this paper we discussed compatible overdetermined as well as underdetermined systems $\E$
of differential equations. We indicated that symmetries generically produce particular
'automodel' solutions to $\E$. But if symmetries are few, they do not allow complete integration.
 % (so in this case nonlocal methods have to invoke).

On the other hand the abundance of symmetries is a sign of integrability. In Section \ref{S3}
we briefly indicated via examples how large algebras of symmetries make models unique.
These symmetries often do not live on the equation-manifold, and a covering is required to see them.
In this way Darboux integrability \cite{D,F} was recast into the language of group quotients
in \cite{AFV} and new examples were integrated, see \cite{AF}.

Integration in closed form is too restrictive in the context of PDEs, and instead one
considers reductions to ODEs. These latter are not arbitrary in the maximal symmetric cases,
as they also possess symmetries. This restricts the models (like those Monge equations
from \cite{AK}) and gives a method to understand integrability.

%%%%%%%%%%%%%%%%%%%%%%%%%%%%%%%%%%%%%%%%%%%%%%%%%%%%%%%%%%%%%%%%%%%%%%%%%%%%

\end{document}